\documentclass[12pt,reqno]{amsart}

\usepackage{amsmath,amsthm,amssymb}
\usepackage{a4wide}

\def\eps{\varepsilon}

\def\d{\,{\rm d}}
\def\div{{\rm div}}

\def\RR{\mathbb{R}}

\newtheorem{theorem}{Theorem}[section]

\newtheorem{lemma}[theorem]{Lemma}

 \newtheorem{definition}[theorem]{Definition}
\newtheorem{assumption}[theorem]{Assumption}

\newtheorem{remark}[theorem]{Remark}

\begin{document}
\title[Identification of diffusion and absorption coefficients]{Simultaneous identification of diffusion and absorption coefficients in a quasilinear elliptic problem}
\author[H. Egger]{Herbert Egger$^\dag$}
\author[J.-F. Pietschmann]{Jan-Frederik Pietschmann$^\dag$}
\author[M. Schlottbom]{Matthias Schlottbom$^\dag$}
\thanks{$^\dag$Numerical Analysis and Scientific Computing, Department of Mathematics, TU Darmstadt, Dolivostr. 15, 64293 Darmstadt. \\
Email: {\tt $\{$egger,pietschmann,schlottbom$\}$@mathematik.tu-darmstadt.de}}

\begin{abstract}
  In this work we consider the identifiability of two coefficients $a(u)$ and $c(x)$ in a quasilinear elliptic partial differential equation 
  from observation of the Dirichlet-to-Neumann map. We use a linearization procedure due to Isakov \cite{Isakov93} and special singular solutions to first determine $a(0)$ and $c(x)$ for $x \in \Omega$. Based on this partial result, we are then able to determine $a(u)$ for $u \in \RR$ by an adjoint approach.
\end{abstract}

\maketitle




\section{Introduction}
We consider the simultaneous identification of two unknown coefficients $a=a(u)$ and $c=c(x)$ in the quasilinear elliptic problem
\begin{eqnarray}
  -\div(a(u)\nabla u) + c u & = & 0 \quad \text{in } \Omega, \label{eq:quasi}\\
			  u & = & g \quad \text{on }\partial \Omega, \label{eq:quasi_bc}
\end{eqnarray}
where $\Omega\subset\RR^n$, $n\in\{2,3\}$, is a bounded, sufficiently regular domain. 
We assume to have access to the \emph{Dirichlet-to-Neumann} map, given by
\begin{align*}
  \Lambda_{a,c}: \ g \mapsto a(u) \partial_n u,
\end{align*}
with $u$ denoting the solution to \eqref{eq:quasi}--\eqref{eq:quasi_bc} with Dirichlet boundary datum $g$. 
The main contribution of our manuscript is the following
\begin{theorem}\label{thm:main}
Let Assumption~\ref{ass:1} hold and assume $\Lambda_{a_1,c_1}=\Lambda_{a_2,c_2}$. Then $a_1(u) = a_2(u)$ and $c_1(x) = c_2(x)$ for all $u\in\RR$ and a.e. $x\in\Omega$. 
\end{theorem}

Let us put this result into perspective: 
Much of the work about identification of unknown coefficients in elliptic and parabolic partial differential equations goes back to the seminal paper of Calder{\'o}n \cite{Calderon80}. 
There,  $c\equiv0$ and the goal is to reconstruct an unknown spatially varying conductivity $a=a(x)$ from observation of the Dirichlet-to-Neumann map. The Calder{\'o}n problem has been studied intensively by many authors, e.g.,  \cite{Astala06,Druskin82,KohnVogelius84,Kohn85,Nachmann96,SylvesterUhlmann87}. 
Indeed, several new technical tools have been developed with this application in mind. 
For a comprehensive review, we refer the reader to \cite{Uhlmann2009}. 
While the question of identifiability of one spatially varying coefficient can be answered  affirmatively under rather general assumptions, the simultaneous determination of two coefficients $a=a(x)$ and $c=c(x)$ is,
in general, not possible, see \cite{ArrLio98}. If $c$ has non-vanishing imaginary part, however, \cite{Grinberg00} provides a local uniqueness result. More recently \cite{Harrach09}, the unique determination of two parameters $a=a(x)$ and $c=c(x)$ was established in the class of piecewise constant and piecewise analytic coefficients, respectively.
Semilinear elliptic equations with $a\equiv1$ and $c=c(x,u)$ have been considered in \cite{IsakovSylvester94}; the case $c=c(u,\nabla u)$ is treated in \cite{Isakov01}.
For quasilinear elliptic equations Sun \cite{Sun96} proved uniqueness of a scalar coefficient $a=a(x,u)$ assuming that $c \equiv 0$. 
In \cite{SunUhlmann97}, this result was generalized to positive definite symmetric matrices $a=a(x,u)\in\RR^{n\times n}$, $n\geq 2$.
Recently, many authors considered the question of uniqueness employing only partial data on the boundary, cf. \cite{Harrach09,ImanuvilovUhlmannYamamoto10,ImanuvilovUhlmannYamamoto11,Isakov01}.
Let us also mention the Bukhgeim-Klibanov method of Carleman estimates introduced in \cite{BughkeimKlibanov81} to prove global uniqueness results for various types of differential equations even in case of non-overdetermined data or single measurements, see \cite{Klibanov13} for a review of this method and a comprehensive list of applications.
Besides uniqueness, also stability issues have been considered in the literature. In this context, let us refer to the work of Alessandrini~\cite{Alessandrini90} and also to \cite{Klibanov13,KlibanovTimonov04}.
Uniqueness results for other types of problems, e.g., of parabolic type or in nonlinear elasticity can be found in \cite{CannonYin89,DuChateauRundell85,DuChateau04,Isakov89,Isakov93,Klibanov04,PilantRundell90} and \cite{KangNakamura02,NakamuraSun94}.
A broad overview over inverse problems for partial differential equations and many more results and references can be found in the book of Isakov \cite{Isakov06}.

The rest of the paper, which is devoted to the proof of Theorem~\ref{thm:main}, is organized as follows:
In Section~\ref{sec:DtN}, we prove well-posedness of \eqref{eq:quasi}--\eqref{eq:quasi_bc},
and we rigorously define the Dirichlet-to-Neumann map.
In Section~\ref{sec:identify_a0}, we first utilize a linearization procedure and show by contradiction that $a(0)$ is uniquely determined by the Dirichlet-to-Neumann map. 
Using the knowledge of $a(0)$, we then obtain the identifiability of $c(x)$ by well-known results for the linearized problem.
The identifiability of $a(u)$ for $u \ne 0$ is established in Section~\ref{sec:identify_a}, and we conclude with a short discussion about possible extensions of our results.

\section{Preliminaries}\label{sec:DtN}

Throughout the rest of the paper, we make the following assumption on the regularity of the domain and the coefficients.

\begin{assumption}\label{ass:1}
  $\Omega\subset\RR^n$ is a bounded domain in two or three space dimensions and $\partial\Omega$ is piecewise $C^1$. Furthermore, we assume that $a\in C^0(\RR)$ and $c\in L^\infty(\Omega)$ such that
\begin{align*}
      \alpha \leq a(u) \leq \frac{1}{\alpha}\quad \text{for all } u\in\RR\quad\text{and}\quad 0\leq c(x) \leq \frac{1}{\alpha}\quad \text{a.e. in } \Omega
\end{align*}
for some constant $\alpha>0$. 
\end{assumption}

We denote by $H^1(\Omega)$ the usual Sobolev space of square integrable functions with square integrable weak derivatives. Functions $u\in H^1(\Omega)$ have well-defined traces $\left. u\right|_{\partial\Omega}$ and we denote by $H^{1/2}(\partial\Omega)$ the space of traces of functions in $H^1(\Omega)$ with norm
\begin{align*}
  \|g\|_{H^{1/2}(\partial\Omega)} = \inf_{u\in H^1(\Omega); \left.u\right|_{\partial \Omega}=g}\|u\|_{H^1(\Omega)}.
\end{align*}
The topological dual space of $H^{1/2}(\partial\Omega)$ is denoted by $H^{-1/2}(\partial\Omega)$. 

For some of our arguments, we will transform the quasilinear equation \eqref{eq:quasi} into a semilinear one.
To do so, let us introduce the primitive function 
\begin{align*}
  A:\RR\to\RR,\quad A(u) = \int_0^u a(\tilde u)\d \tilde u,
\end{align*}
which is monotonically increasing and differentiable. Since we assumed that 
$a\geq \alpha >0$, the function $A$ is one-to-one and onto, and we can define its inverse $H:\RR\to \RR$, $H(U)=A^{-1}(U)$ with derivative
\begin{align*}
  \frac{1}{\alpha} \ge H'(U) = \frac{1}{a(H(U))} \geq \alpha>0.
\end{align*}
For any weak solution $u$ of \eqref{eq:quasi}--\eqref{eq:quasi_bc}, 
the function $U=A(u)$ then solves the boundary value problem
\begin{eqnarray}
  -\Delta U + c H(U) &=& 0 \, \qquad\text{in }\Omega, \label{eq:semi} \\
		  U & =& G \qquad\text{on }\partial\Omega,  \label{eq:semi_bc}
\end{eqnarray}
with boundary datum $G = A(g)$. 
Note that by our assumption on the coefficients $u=H(U)\in H^1(\Omega)$ whenever $U\in H^1(\Omega)$; this follows easily from the monotonicity and differentiability of $H$ and using the chain rule for Sobolev functions \cite{GT}.
The next theorem establishes the well-posedness of the problems \eqref{eq:quasi}--\eqref{eq:quasi_bc} and \eqref{eq:semi}--\eqref{eq:semi_bc}, respectively.
\begin{theorem}\label{thm:existence}
  Let Assumption~\ref{ass:1} hold. Then for every $g\in H^{1/2}(\partial\Omega)$ there exists a unique solution $u\in H^1(\Omega)$ to \eqref{eq:quasi}--\eqref{eq:quasi_bc} which satisfies the a-priori estimates
  \begin{align}
    \|u\|_{H^1(\Omega)} &\leq C \|g\|_{H^{1/2}(\partial\Omega)}, \label{eq:apriori}
  \end{align}
with a constant $C$ depending only on $\alpha$ and $\Omega$.
\end{theorem}
We could not find a reference for this result, so we sketch the proof.
\begin{proof}
Let us first establish the existence of a solution:
  Given $\tilde u \in L^2(\Omega)$, consider the linear boundary value problem 
  \begin{eqnarray*}
    -\div(a(\tilde u)\nabla u) + c u & = & 0 \quad \text{in } \Omega,\\
			    u & = & g \quad \text{on }\partial \Omega.
  \end{eqnarray*}
  Since $\tilde u$ is measureable, so is $a(\tilde u)$, and since $a(\tilde u)\ge \alpha>0$, the existence of a unique solution $u \in H^1(\Omega)$ is ensured by the Lax-Milgram lemma, cf. \cite[Theorem~5.8]{GT}. Moreover, we have $\|u\|_{H^1(\Omega)} \le C \|g\|_{H^{1/2}(\partial \Omega)}$ with $C$ only depending on $\alpha$ and $\Omega$. 
  Next, consider the nonlinear operator $T:L^2(\Omega)\to L^2(\Omega)$ defined by $T\tilde u:=u$ with $u$ the solution of the problem above. We will establish the existence of a fixed-point for the mapping $T$, which then is a solution of \eqref{eq:quasi}--\eqref{eq:quasi_bc}, by a compactness argument:
  Due to the a-priori estimate for the linear problem, $T$ maps the compact convex set $M=\{v \in L^2(\Omega) : \|v\|_{H^1(\Omega)} \le C \|g\|_{H^{1/2}(\partial \Omega)}\}$ into itself. Moreover, $T$ is continuous, which can be seen as follow: observe that $\tilde u_n \to \tilde u$ in $L^2(\Omega)$ implies that $\tilde u_{n_k} \to \tilde u$ a.e. for some subsequence $\tilde u_{n_k}$. 
  By Assumption~\ref{ass:1} and Lebesgue's dominated convergence theorem, we get $a(\tilde u_{n_k}) \nabla u \to a(\tilde u) \nabla u$ in $L^2(\Omega)$. Together with the a-priori estimate for the linear problem, this yields the continuity of $T$; see also the proof of Lemma~\ref{lem:linearize}.
  The existence of a fixed-point for $T$ in $M$ then follows by Schauder's fixed-point theorem \cite[Theorem~11.1]{GT}. 
  Clearly, any regular fixed-point of $T$ is also a solution of \eqref{eq:quasi}--\eqref{eq:quasi_bc} and the a-priori estimate follows from the definition of the set $M$.

  Let us now turn to the question of uniqueness: 
  Assume that there exist two solutions $u_1$, $u_2$ to \eqref{eq:quasi}--\eqref{eq:quasi_bc} with the same Dirichlet boundary data and set $U_1 = A(u_1)$ and $U_2= A(u_2)$. Then, $U=U_1-U_2$ solves
  \begin{align*}
    -\Delta U + c H'(\xi(x)) U &= 0 \quad \text{ in }\Omega,\\
	U&=0 \quad \text{ on } \partial \Omega,
  \end{align*}
  where we used $H(U_1(x))-H(U_2(x)) = H'(\xi(x))(U_1(x)-U_2(x))$ a.e. for some measureable function $\xi(x)$. Since $H'\geq 0$, we obtain from the weak maximum principle \cite[Theorem~8.1]{GT} that $U \equiv 0$, and by monotonicity of $A$ we deduce that $u_1=u_2$.
\end{proof}

To give a precise definition of the Dirichlet-to-Neumann map in our functional setting, we introduce for $u\in H^1(\Omega)$ the generalized co-normal derivative $a(u)\partial_n u\in H^{-1/2}(\partial\Omega)$ as in \cite{McLean} by
\begin{align}\label{eq:conormal}
\langle a(u) \partial_n u, \left.v\right|_{\partial\Omega}\rangle = \int_\Omega a(u)\nabla u\nabla v + c uv \d x, \quad v\in H^1(\Omega).
\end{align}
Here, $\langle\cdot,\cdot\rangle$ denotes the duality pairing of $H^{-1/2}(\partial\Omega)$ and $H^{1/2}(\partial\Omega)$. Note that this definition coincides with the usual definition of the co-normal derivative if $u$ is a solution of \eqref{eq:quasi}. This motivates the following
\begin{definition} 
For any pair of coefficients $a$ and $c$ satisfying Assumption~\ref{ass:1}, we define the \emph{Dirichlet-to-Neumann map}
\begin{align*}
  \Lambda_{a,c}: H^{1/2}(\partial\Omega) \to H^{-1/2}(\partial\Omega),\quad g \mapsto a(u) \partial_n u,
\end{align*}
where $u\in H^1(\Omega)$ is the solution of \eqref{eq:quasi}--\eqref{eq:quasi_bc} with boundary value $g$.
\end{definition}

After establishing the well-posedness of the governing boundary value problem and defining the 
Dirichlet-to-Neuman map rigorously, we can now start to investigate the inverse problem of identifying the coefficients $a$ and $c$.

\section{Uniqueness of $a(0)$ and $c$}\label{sec:identify_a0}

Following an idea of Isakov \cite{Isakov93}, we employ a linearization strategy to obtain uniqueness for $a(0)$ and $c(x)$. 
Consider the following linear boundary value problem
\begin{eqnarray}
  -a(0)\Delta v + c v & = & 0\ \quad \text{in } \Omega, \label{eq:linear}\\
		  v & = & g^* \quad \text{on }\partial \Omega. \label{eq:linear_bc}
\end{eqnarray}
The existence of a unique weak solution $v \in H^1(\Omega)$ follows again from the Lax-Milgram Theorem. The Dirichlet-to-Neumann map associated with the linear problem is given by
\begin{align*}
  \Lambda^*_{a(0),c}: H^{1/2}(\partial \Omega) \to H^{-1/2}(\partial\Omega), \quad g^*\mapsto a(0)\partial_n v,
\end{align*}
where $v$ is the solution of \eqref{eq:linear}--\eqref{eq:linear_bc} with boundary datum $g^*$.
With similar arguments as in \cite{Isakov93}, we obtain
\begin{lemma}\label{lem:linearize}
  The Dirichlet-to-Neumann map $\Lambda_{a,c}$ for \eqref{eq:quasi}--\eqref{eq:quasi_bc} determines the Dirichlet-to-Neumann map $\Lambda^*_{a(0),c}$ associated with \eqref{eq:linear}--\eqref{eq:linear_bc}.
\end{lemma}
\begin{proof}
Let $g^* \in H^{1/2}(\partial \Omega)$ be given. 
For any $\tau\in\RR$, we denote by $u_\tau$ the solution of \eqref{eq:quasi}--\eqref{eq:quasi_bc} with boundary value $\tau g^*$. By Theorem~\ref{thm:existence}, such a solution $u_\tau$ exists and is unique, and  for $\tau = 0$ we have $u_0 \equiv 0$. 
The function $v_\tau := (u_\tau-u_0)/\tau = u_\tau / \tau$ then is a solution of
  \begin{eqnarray*}
  -\div(a(u_\tau)\nabla v_\tau) + c v_\tau & = & 0\ \quad \text{in } \Omega,\\
		  v_\tau & = & g^* \quad \text{on }\partial \Omega. 
  \end{eqnarray*}
  Moreover, with $v$ defined by \eqref{eq:linear}--\eqref{eq:linear_bc}, the difference $w_\tau=v-v_\tau$ solves
  \begin{eqnarray*}
  -\div(a(u_\tau)\nabla w_\tau) + c w_\tau & = & -\div((a(u_\tau)-a(0))\nabla v) \quad \text{in } \Omega,\\
		  w_\tau & = & 0 \quad \text{on }\partial \Omega. 
  \end{eqnarray*}
  Using standard a-priori estimates for linear elliptic problems and Assumption~\ref{ass:1}, we obtain
  \begin{align*}
    \|w_\tau\|_{H^1(\Omega)} \leq C \|(a(u_\tau)-a(0))\nabla v\|_{L^2(\Omega)},
  \end{align*}
  with a constant $C$ depending only on $\alpha$ and $\Omega$.
  Using the a-priori estimate \eqref{eq:apriori}, we obtain $u_\tau\to 0$ in $H^1(\Omega)$ as $\tau\to 0$, and hence, by a subsequence argument, $u_\tau(x)\to 0$ as $\tau\to 0$ for a.e.\@ $x\in\Omega$. By continuity of the parameter, it follows that $a(u_\tau(x))\to a(0)$ for a.e.\@  $x\in\Omega$, and from Lebesgue's dominated convergence theorem, we infer that $w_\tau\to 0$ in $H^1(\Omega)$ as $\tau\to 0$. Using the definition of the co-normal derivative \eqref{eq:conormal}, we further obtain
  \begin{align*}
   \frac{1}{\tau}\Lambda_{a,c} \tau g^* = a(u_\tau)\partial_n v_\tau \to a(0)\partial_n v= \Lambda^*_{a(0),c} g^*
  \end{align*}
  in $H^{-1/2}(\partial\Omega)$ as $\tau \to 0$, and hence $\Lambda^*_{a(0),c}$ is determined by $\Lambda_{a,c}$.
\end{proof}
As a next step, we turn to the identification of $a(0)$ and $c(x)$ from knowledge of the linearized Dirichlet-to-Neumann map $\Lambda^*_{a(0),c}$. 
Let $(a_1, c_1)$ and $(a_2,c_2)$ satisfy Assumption~\ref{ass:1} and denote by $v_1$ and $v_2$ the corresponding solutions of \eqref{eq:linear}--\eqref{eq:linear_bc} with coefficients $(a_1(0),c_1)$ and $(a_2(0),c_2)$, respectively. 
The definition of the co-normal derivative yields the following orthogonality relation
\begin{align}\label{eq:orth_linear}
 & \langle \big(\Lambda^*_{a_1(0),c_1} - \Lambda^*_{a_2(0),c_2}\big)g^*,  g^*\rangle \\
 &\qquad \qquad = \int_{\Omega} (a_1(0)-a_2(0)) \nabla v_1 \nabla v_2 + (c_1-c_2) v_1 v_2 \d x. \notag
\end{align}
We are now in a position to prove the following
\begin{theorem}\label{thm:identify_a0}
  If $\Lambda^*_{a_1(0),c_1}= \Lambda^*_{a_2(0),c_2}$ then $a_1(0)=a_2(0)$.
\end{theorem}
\begin{proof}
The proof is inspired by the construction of singular solutions utilized in \cite{Alessandrini90}. Let $\Phi_y(x)$ be the fundamental solution for the Laplace equation, i.e., we have $\Phi_y(x)= 1/|x-y|$ for $n=3$ and $\Phi_y(x)=\log(|x-y|)$ for $n=2$. Note that for any $y \in \RR^n$ we have $\Phi_y\in L^2(\Omega)$ while $\Phi_y \in H^1(\Omega)$, if, and only if, $y \notin \overline{\Omega}$. 
  Now suppose that $a_1(0)\neq a_2(0)$ and let $w_i\in  H^{1}_0(\Omega)$, $i=1,2$, be the solution of 
  \begin{align*}
  -a_i(0)\Delta w_i + c_i w_i & =  - c_i \Phi_y \, \quad \text{in } \Omega,\\
		  w_i & =  0\qquad\quad \text{on }\partial \Omega.
  \end{align*}
  The function $v_i=w_i + \Phi$ then is a solution of \eqref{eq:linear}--\eqref{eq:linear_bc} with $g^*=\Phi$, and we see that $\|v_1 v_2\|_{L^1(\Omega)} \le C$ for all $y \in \RR^n$, 
  but $\|\nabla v_1\nabla v_2\|_{L^1(\Omega)} < \infty$ only if $y \notin \overline{\Omega}$. Inserting $v_1$ and $v_2$ into the orthogonality relation \eqref{eq:orth_linear} and rearranging terms, we obtain
  \begin{align*}
      (a_1(0)-a_2(0)) \int_{\Omega} \nabla v_1 \nabla v_2 \d x= \int_\Omega (c_2-c_1) v_1 v_2 \d x.
  \end{align*}
  Since the integral on the right-hand side is uniformly bounded, but that on the left-hand side diverges as ${\rm dist}(y,\partial \Omega)\to 0$, we arrive at a contradiction. Hence, $a_1(0) = a_2(0)$.
\end{proof}

Once $a(0)$ is determined, the uniqueness of $c(x)$ follows from known results: The three-dimensional case can be found in \cite{SylvesterUhlmann87} or \cite[Theorem~5.2.2]{Isakov06}.
For $0\leq c\in L^\infty(\Omega)$, the uniqueness result for $n=2$ can be deduced from the uniqueness of the conductivity problem \cite{Astala06}, see \cite[~Corollary~5.5.2]{Isakov06}. The restriction $c\geq 0$ can possibly be relaxed using the results of \cite{Bukhgeim2008,ImanuvilovYamamoto2012}.
Thus we obtain
\begin{theorem}\label{thm:identify_a0_c}
   Assume that $\Lambda^*_{a_1(0),c_1}= \Lambda^*_{a_2(0),c_2}$. Then $a_1(0)=a_2(0)$ and $c_1(x)=c_2(x)$ for a.e. $x \in \Omega$.
\end{theorem}

\section{Identification of $a$}\label{sec:identify_a}
To show the uniqueness of  $a(u)$ for $u \ne 0$, we translate the techniques of the previous section to the nonlinear problem.
By the definition of the co-normal derivative \eqref{eq:conormal}, there holds
\begin{align*}
  \langle \Lambda_{a,c} g, \left.\lambda\right|_{\partial\Omega} \rangle = \int_{\Omega} a(u) \nabla u\nabla \lambda + cu \lambda\d x
\end{align*}
for any function $\lambda\in H^1(\Omega)$. 
Subtracting this identity for two pairs $(a_1,c)$ and $(a_2,c)$ of admissible parameters,
and using $\nabla( A_i(u(x)) = A_i'(u(x)) \nabla u = a_i(u(x)) \nabla u,\,i=1,2$ and integration by parts we get
\begin{align*}
  \langle (\Lambda_{a_1,c}-\Lambda_{a_2,c}) g, \left.\lambda\right|_{\partial\Omega} \rangle 
  & = \int_\Omega \left( A_1(u_1)-A_2(u_2) \right)\left(-\Delta\lambda\right) + c(u_1-u_2)\lambda \d x \\
  & + \int_{\partial\Omega} \left( A_1(g) - A_2(g) \right) \partial_n \lambda \d s
\end{align*}
To simplify this expression, we consider only test functions $\lambda$ which are solutions of
\begin{eqnarray}
  -\Delta\lambda &=& 0        \ \  \quad\text{in } \Omega,\label{eq:adjoint}\\
			       \lambda &=& \lambda_D \quad\text{on }\partial\Omega \label{eq:adjoint_bc}
\end{eqnarray}
for some appropriate boundary datum $\lambda_D$, which yields 
%
\begin{align}\label{eq:orth_nonlinear}
 \langle\big(\Lambda_{a_1,c}-\Lambda_{a_2,c}\big) g, \lambda_D \rangle =  \int_\Omega c(u_1-u_2)\lambda \d x + \int_{\partial\Omega} \left( A_1(g) - A_2(g) \right) \partial_n \lambda \d s.
\end{align}
Note that the left hand side will vanish, if the Dirichlet-to-Neumann maps coincide. 
We can therefore retrieve information about $a_1-a_2$, by choosing a suitable function $\lambda$ satisfying \eqref{eq:adjoint}--\eqref{eq:adjoint_bc}. 
%
\begin{theorem}\label{thm:identify_a}
  Let $\Lambda_{a_1,c}=\Lambda_{a_2,c}$ for some $a_1$, $a_2$ and $c$ satisfying Assumption~\ref{ass:1}. 
  Then $a_1(u) = a_2(u)$ for all $u\in\RR$.
\end{theorem}
\begin{proof}
We only consider the three dimensional case and assume, for simplicity, that the boundary $\partial \Omega$ of the domain is 
flat near some point $\bar x \in \partial \Omega$. 
Suppose there exists $\bar{g}\in \RR$ with $a_1(\bar g) - a_2(\bar g) > 0$. 
Then by continuity, $a_1(u) < a_2(u)$ for $u \in [\underline g, \bar g]$ with $\underline g < \bar g$.
Let us define the boundary datum $g$ by 
$$
g(x) = \left\{\begin{array}{ll} 
\bar g, & |x-\bar x| \le r, \\ 
\frac{|x-\bar x|-r}{s-r} \underline g  + \frac{s-|x-\bar x|}{s-r} \bar g, & r < |x-\bar x|< s, \\ 
\underline g, & |x-\bar x| \ge s,
\end{array}\right.
$$
where $0<r<s$ are sufficiently small and will be specified below.
For $\eps>0$ define $\lambda^\eps(x) = n(\bar x) \cdot \nabla \Phi_{y^\eps}(x)$ with 
$y^\eps=\bar x + \eps n(\bar x)$ and $\Phi_y(x) = 1/|x-y|$ as in the proof of Theorem~\ref{thm:identify_a0}.
Observe that $\lambda^\eps$ is harmonic in $\Omega$ and uniformly bounded in $L^1(\Omega)$ for all $\eps \ge 0$.
Now from \eqref{eq:orth_nonlinear} and $\Lambda_{a_1,c}=\Lambda_{a_2,c}$, we obtain 
\begin{align} 
 -\int_\Omega c(u_1-u_2)\lambda^\eps \d x 
&= \int_{\partial\Omega} \big(A_1(g) - A_2(g) \big) \partial_n \lambda^\eps \d s \label{eq:contradiction}\\
&= \big(A_1(\underline{g}) - A_2(\underline{g})\big) \int_{\partial \Omega} \partial_n \lambda^\eps \d s + \int_{\partial \Omega} B(g) \partial_n \lambda^\eps \d s. \nonumber
\end{align}
Since $\lambda^\eps$ is harmonic in $\Omega$, the first integral on the right hand side vanishes
and in the second term we abbreviated 
\begin{align*}
B(u) 
&=  \big(A_1(u) -A_1(\underline g)\big) - (A_2(u) - A_2(\underline g)\big) 
= \int_{\underline g}^{u} a_1(u) - a_2(u) \d u.
\end{align*}
Since $a_1(u) - a_2(u)>0$, the function $B(u)$ is strictly monotonically increasing and positive on $(\underline g,\bar g]$.
The second term can then be further evaluated by 
\begin{align*}
 \int_{\partial\Omega} B(g) \partial_n \lambda^\eps \d s 
&= B(\bar g)\int_{|x-\bar x| < r}  \partial_n \lambda^\eps \d s + \int_{r < |x-\bar x| < s} B(g) \partial_n \lambda^\eps \d s \\
&= -B(\bar g) \frac{r^2}{(r^2+\eps^2)^{3/2}} + \tilde{B} \left( \frac{r^2}{(r^2+\eps^2)^{3/2}} - \frac{s^2}{(s^2+\eps^2)^{3/2}} \right),
\end{align*}
where integration is performed over subsets of the boundary and $\tilde B \in [0,B(\bar g)]$. 
This formula holds for all $0<r<s$ sufficiently small and all $\eps > 0$. 
By choosing $r=\eps$ and $s = \eps + \eps^3$ and letting $\eps \to 0$, 
the first integral can be made arbitrarily large while the second integral can be made arbitrarily small. 
Since the left hand side of \eqref{eq:contradiction} is uniformly bounded as $\eps \to 0$, we obtain a contradiction
to the assumption that $a_1(\bar g) - a_2(\bar g)>0$. 
The two dimensional case and curved boundaries can be treated with similar arguments.
\end{proof}

\begin{remark}  \rm
Similar orthogonality relations and adjoint problems have been used for one-dimensional equations before. 
In \cite{DuChateau04}, the identifiability of $a$ is established by  controlling the sign of $u_{1x}$ and $\lambda_x$, which is possible with monotonicity arguments in the one-dimensional case. This argument is however not applicable in the multi-dimensional case. 
\end{remark}

\noindent
Summarizing the previous results, we obtain the \\
{\bf Proof of Theorem \ref{thm:main}:} 
If $\Lambda_{a_1,c_1}=\Lambda_{a_2,c_2}$, then Lemma \ref{lem:linearize} implies that $\Lambda^*_{a_1(0),c_1}=\Lambda^*_{a_2(0),c_2}$. Thus $a_1(0)=a_2(0)$ by Theorem~\ref{thm:identify_a0} and $c_1=c_2$ by Theorem \ref{thm:identify_a0_c}. 
The assertion  $a_1(u)=a_2(u)$ follows from Theorem \ref{thm:identify_a}, which concludes the proof.
\qed

\section{Discussion}


Concerning stability when reconstructing $c$ the best one can expect is an estimate of logarithmic type even if we assume that the coefficient $a$ is known and constant; see \cite{Alessandrini90} for details. Thus, the inverse problem considered in this paper is severely ill-posed.

\section*{Acknowledgments}
HE acknowledges support by DFG via Grant IRTG 1529 and GSC 233.
The work of JFP was supported by DFG via Grant 1073/1-1 and from the Daimler and Benz Stiftung via Post-Doc Stipend 32-09/12. 
We would like to thank Prof. Bastian von Harrach for valuable comments when completing the proof of Theorem~\ref{thm:identify_a}.

\end{document}